\documentclass[11pt,a4paper]{amsart}
\usepackage[english]{babel}
\usepackage{pst-grad} % For gradients
\usepackage{pst-plot} % For axes
\usepackage{pstricks}
\usepackage{amsmath,amssymb,mathrsfs,amsthm}
\usepackage[dvips]{graphicx}
\usepackage{epsfig}
\newcommand{\RR}{\mathbb R}

\newcommand{\ZZ}{\mathbb Z}

\newcommand{\T}{\mathbb T}
\newcommand{\TT}{\mathbb T}
\newcommand{\pat}{\partial_t}
\newcommand{\pax}{\partial_x}

\newcommand{\jeps}{\mathcal{J}_\epsilon*}

\newcommand{\vertiii}[1]{{\left\vert\kern-0.25ex\left\vert\kern-0.25ex\left\vert #1 
    \right\vert\kern-0.25ex\right\vert\kern-0.25ex\right\vert}}

\usepackage[pdfauthor={Rafael Granero-Belinchon},pdftitle={...},bookmarks,colorlinks]{hyperref}

\newcounter{comentcount}
\newcounter{teocount}
\newtheorem{lem}{Lemma}
\newtheorem{prop}{Proposition}
\newtheorem{corol}{Corollary}
\newtheorem{teo}[teocount]{Theorem}  
\newtheorem{defi}{Definition}

\newtheorem{remark}{Remark}

\title{Global solutions for a supercritical drift-diffusion equation}

\author[J. Burczak]{Jan Burczak}
\email{jb@impan.pl}
\address{Institute of Mathematics of the Polish Academy of Sciences, Warsaw, 21 00-956, Poland}

\author[R. Granero-Belinch\'{o}n]{Rafael Granero-Belinch\'{o}n}
\email{rgranero@math.ucdavis.edu}
\address{Department of Mathematics, University of California, Davis, CA 95616, USA}

\begin{document}

\begin{abstract}
We study the global existence of solutions to a one-dimensional drift-diffusion equation with logistic term, generalizing the classical parabolic-elliptic Keller-Segel aggregation equation arising in mathematical biology. In particular, we prove that there exists a global weak solution, if the order of the fractional diffusion $\alpha \in (1-c_1, 2]$, where $c_1>0$ is an explicit constant depending on the physical parameters present in the problem (chemosensitivity and strength of logistic damping). Furthermore, in the range $1-c_2<\alpha\leq 2$ with $0<c_2<c_1$, the solution is globally smooth. Let us emphasize that when $\alpha<1$, the diffusion is in the supercritical regime.
\end{abstract}

\maketitle 

\tableofcontents
\section{Introduction}
The drift-diffusion equation 
\begin{equation}\label{eqDD}
\pat u=-\nu \Lambda^\alpha u+\nabla\cdot(u B(u))+f(u),
\end{equation}
where $B(u)$  is typically a vector of nonlocal operators and $\Lambda =\sqrt{-\Delta}$ (see \eqref{Lfourier} below), appears widely in applications. The parameter $0\leq \alpha\leq 2$ is the order of the diffusion and it measures the strength of the viscous effects.

For instance, the two dimensional incompressible Navier-Stokes equations in its vorticity formulation can be written as
\begin{equation}\label{eqvorticity}
\pat u=\Delta u+\nabla\cdot(u \nabla^\perp \Delta^{-1}u),
\end{equation}
where $\nabla^\perp=(-\partial_{x_2},\partial_{x_1})$. This equation governs the motion of two-dimensional, incompressible, homogeneous fluids in absence of forcing (see \cite{bertozzi-Majda}). Equation \eqref{eqvorticity} can be recovered from equation \eqref{eqDD} by taking  $\nu=1$, $\alpha=2$, $f=0$ and $B=\nabla^\perp \Delta^{-1}$.

Another famous equation akin to \eqref{eqDD} is the parabolic-elliptic simplification of the Keller-Segel system with logistic source
\begin{equation}\label{eqPKS}
\pat u=\Delta u+\nabla\cdot(u \nabla\Delta^{-1}u)+\mu u -ru^2,
\end{equation}
It appears as a model of \emph{chemotaxis}, \emph{i.e.} the proliferation and motion of cells (see the pioneer work of E. Keller \& L. Segel \cite{keller1970initiation} and the  reviews by A. Blanchet \cite{blanchet2011parabolic} and Hillen \& Painter \cite{Hillen3}, whose cell-kinetics model (M8) is the doubly-parabolic version of \eqref{eqPKS}). Here $u\geq0$ is the density of cells. To obtain \eqref{eqPKS} form \eqref{eqDD} one takes $\nu=1$, $\alpha=2,$ $B=\nabla \Delta^{-1}$ and $f=\mu u -ru^2$. Equation \eqref{eqPKS} is also a model of gravitational collapse (see the works by  Biler \cite{Bi1} and Ascasibar, Granero-Belinch\'on \& Moreno \cite{AGM})

Furthermore, in one spatial dimension, the equation
\begin{equation}\label{eqdislocation}
\pat u=-\pax(uHu),
\end{equation}
where $H$ denotes the Hilbert transform (see \eqref{Hfourier} and \eqref{Hkernel} below), has been proposed as a model of the dynamics of a dislocation density $u$ (see \cite{dislocation1} and the work by Biler, Karch \& Monneau \cite{biler2010nonlinear}). Equation \eqref{eqdislocation} appears also in a totally different context, namely as a one-dimensional model of the surface quasi-geostrophic equation (see Castro \& C\'ordoba \cite{CC}). In order to recover \eqref{eqdislocation} from \eqref{eqDD}, we choose $\nu = 0$, $f=0$ and $B=-H$.

Finally, notice that the famous Burgers equation
\begin{equation}\label{eqBurgers}
\pat u=-\Lambda^\alpha u-u\pax u
\end{equation}
can be obtained from \eqref{eqDD} by taking $B(u)=-u/2$.

In all these equations there is a competition between the diffusion term given by
$$
\Lambda^\alpha
$$
and the transport term 
$$
\nabla\cdot (uB(u)).
$$
Consequently, if the order of the diffusion $\alpha$ is large enough, the solution $u$ will remain smooth for every time. On the other hand, when the order of the diffusion is too weak, there are finite-time singularities. 

Burgers equation \eqref{eqBurgers} is the paradigm here. There, if $\alpha>1$ (this is known as the subcritical case), there exists a global, smooth solution, whereas for $\alpha<1$ (this is known as the supercritical case), there is a finite time $T^*$, such that
$$
\limsup_{t\rightarrow T^*} \|\pax u(t)\|_{L^\infty}=\infty.
$$
The remaining case $\alpha=1$ is known as the critical case. Here the diffusive operator is of order one, so it perfectly matches the order of the transport term. Kiselev, Nazarov \& Shterenberg \cite{kiselevburgers}, Dong, Du \& Li \cite{dong2009finite} and Constantin \& Vicol \cite{constantin2013long} proved that in the critical case, there exists a global, smooth solution. Furthermore, Dabkowski, Kiselev, Silvestre \& Vicol proved the global existence of classical solutions for a logarithmically supercritical Burgers equation \cite{DKSV}.

A similar picture appears for the one-dimensional case of \eqref{eqDD} with $f=0$, which is a fractional-diffusion version of Keller-Segel system \eqref{eqPKS} with $\mu=r=0$. This is because the Burgers equation \eqref{eqBurgers} is a primitive equation to this one. For more details on available results, see Section \ref{sec:2}.  

\subsection{Considered problem}
In this paper we are going to study the drift-diffusion equation \eqref{eqDD} on the one-dimensional torus  $\TT$, i.e.
\begin{equation}\label{eq1}
\pat u=-\Lambda^\alpha u+\chi\pax(u B(u))+f(u),\quad x\in\TT
\end{equation}
 with 
 \begin{equation}\label{eqB}
B(u)=\Lambda^{\beta-1}H(1 + \Lambda^{\beta})^{-1}u,
\end{equation}
and $f$ of logistic form
\begin{equation}\label{eqf}
f(u)=ru(1-u), \quad 0<r.
\end{equation}
Our primary motivation here is the (parabolic-elliptic) Keller-Segel model, which corresponds to $\beta=2$. 

\subsection{Outline of results}
Our results concern the global in time existence of solutions to \eqref{eq1} with $\alpha \le 2$ and $\chi>r$. Primarily, we will show that for positive \[\alpha > 1-r/\chi,\]

equation \eqref{eq1} possesses a global smooth solution, see Theorem \ref{th:globalstrong}. To the best of our knowledge, this is the first result on global smooth solutions in the supercritical regime of an equation modelling chemotaxis. This result may be somehow surprising, because the global existence is obtained due to an inviscid regularization $-r u^2$, present in the logistic term that may be arbitrarily small, i.e. $0<r\ll1$. Let us recall that for $r=0$ the critical case $\alpha=1$ is globally well-posed, see the work \cite{BG3} by the authors. 

Along with the smoothness result, we also study the global existence of weak solutions and show that there exists a global weak solution for 
$$
\alpha>
\begin{cases}
1-  \frac{r}{\chi-r} &\text{ for } \chi/2 > r \cr
0 &\text{ for } \chi/2\leq r
\end{cases}
$$ 
see Theorem \ref{th:globalweak}. At the cost of reducing the integrability and smoothness of our weak solutions, we can relax the condition on $\alpha$ and obtain
$$
\alpha>
\begin{cases}
2\frac{1-s-s^2}{2+s} &\text{ for } \chi/2 > r \text{ and }0<s<\frac{r}{\chi-r}\cr
0 &\text{ for } \chi/2\leq r.
\end{cases}
$$ 

This result is more difficult to obtain than its analogue for the Burgers equation \eqref{eqBurgers}, due to lack of the $L^2$ energy balance. These main theorems cover the case of an arbitrary, finite time of existence $T$ (with no smallness assumptions on data). They are supplemented by Proposition \ref{th:largetime} on boundedness of $u$ for $T=\infty$. It holds for supercritical or critical diffusions, where in the latter case we need an explicit smallness condition on data or chemosensitivity parameter $\chi$.

Let us remark here that our result may be easily adapted for another class of operators $B$. For instance, one may consider the case
\begin{equation}\label{eqB2}
B(u)=\Lambda^{\beta-1}H\Lambda^{-\beta}(u-\langle u \rangle),
\end{equation}
where 
$$
\langle u \rangle=\frac{1}{2\pi}\int_\TT u(x)dx.
$$
Moreover, the same ideas work for the case of the spatial domain being $\RR$.

\subsection{Plan of the paper} 
In Section \ref{sec:2} we review some results on drift-diffusion equation that are related to our problem. Next, in Section \ref{sec:3}, we state our main results. In Section \ref{sec:4} we provide some preliminary results, including the local existence and analyticity and certain estimates for lower order norms. Section \ref{sec:5} is devoted to the proof of Theorem \ref{th:globalweak} on  global in time existence of weak solution, whereas in Section \ref{sec:6} we prove Theorem \ref{th:globalstrong} on  global in time existence of smooth solution.
In Section \ref{sec:7} we prove Proposition \ref{th:largetime} on boundedness for $T=\infty$. Finally, in the appendix we provide some estimates for the fractional laplacian and inequalities for the standard $L\log(L)$ entropy and the $\dot{W}^{\gamma,p}$ seminorm.

\subsection{Notation}
We write $H$ for the Hilbert transform and $\Lambda =\sqrt{-\Delta}$, \emph{i.e.}
\begin{equation}\label{Hfourier}
\widehat{Hu}(\xi)=-i\text{sgn}(\xi)\hat{u}(\xi), 
\end{equation}
\begin{equation}\label{Lfourier}
\widehat{\Lambda^s u}(\xi)=|\xi|^s\hat{u}(\xi),
\end{equation}
where $\hat{\cdot}$ denotes the usual Fourier transform. Notice that in one dimension $\Lambda=\pax H$ and $\widehat{Hu}(0)=0$. These operators have the following kernel representation
\begin{equation}\label{Hkernel}
Hu(x)=\frac{1}{\pi}\text{P.V.}\int_{\TT}\frac{u(y)}{\tan\left(\frac{x-y}{2}\right)}dy,
\end{equation}
\begin{equation}\label{Lkernel}
\Lambda u(x)=\frac{1}{2\pi}\text{P.V.}\int_{\TT}\frac{u(x)-u(y)}{\sin^2\left(\frac{x-y}{2}\right)}dy.
\end{equation}
We also have
\begin{align}\label{Lkernelalpha}
\Lambda^\alpha u(x)&=c_\alpha\bigg{(}\sum_{k\in \ZZ, k\neq 0}\int_{\TT}\frac{u(x)-u(x-\eta)d\eta}{|\eta+2k\pi|^{1+\alpha}}\nonumber\\
&\quad +\text{P.V.}\int_{\TT}\frac{u(x)-u(x-\eta)d\eta}{|\eta|^{1+\alpha}}\bigg{)},
\end{align}
with
$$
c_\alpha=\frac{\Gamma(1+\alpha)\cos((1-\alpha)\pi/2)}{\pi}\geq0.
$$
We denote by $\mathcal{J}_\epsilon$ the periodic heat kernel at time $t=\epsilon$. 

We write $T_{max}$ for the maximum time of existence of a given solution $u$.

Finally, to simplify the notation, for the case $B(u)$ given by \eqref{eqB} we define
\begin{equation}\label{eqvB}
v=(1+\Lambda^\beta)^{-1}u,
\end{equation}
while in the case \eqref{eqB2}
\begin{equation}\label{eqvB2}
v=\Lambda^{-\beta}(u-\langle u\rangle).
\end{equation}
Notice that, in both cases, $v$ can be defined using Fourier series. Furthermore, we have
$$
B(u)=\Lambda^{\beta-1}Hv.
$$

\subsection{Functional spaces}
The fractional $L^p$-based Sobolev spaces, $W^{s,p}(\TT)$, are 
$$
W^{s,p}=\left\{f\in L^p(\TT), \pax^{\lfloor s\rfloor} f\in L^p(\TT), \frac{|\pax^{\lfloor s\rfloor}f(x)-\pax^{\lfloor s\rfloor}f(y)|}{|x-y|^{\frac{1}{p}+(s-\lfloor s\rfloor)}}\in L^p(\TT\times\TT)\right\},
$$
with norm
$$
\|f\|_{W^{s,p}}^p=\|f\|_{L^p}^p+\|f\|_{\dot{W}^{s,p}}^p, 
$$
$$
\|f\|_{\dot{W}^{s,p}}^p=\|\pax^{\lfloor s\rfloor} f\|^p_{L^p}+\int_{\TT}\int_{\TT}\frac{|\pax^{\lfloor s\rfloor}f(x)-\pax^{\lfloor s\rfloor}f(y)|^p}{|x-y|^{1+(s-\lfloor s \rfloor)p}}dxdy.
$$
In the case $p=2$, we write $H^s(\TT)=W^{s,2}(\TT)$ for the standard non-homogeneous Sobolev spaces with norm
$$
\|f\|_{H^s}^2=\|f\|_{L^2}^2+\|f\|_{\dot{H}^s}^2, \quad \|f\|_{\dot{H}^s}=\|\Lambda^s f\|_{L^2}.
$$

\section{Prior results for one-dimensional drift-diffusion equations}\label{sec:2}
In this section we are going to review some results concerning certain one-dimensional models. 

First, let us recall the following nonlocal version of the Kuramoto-Shivashinsky equation:
\begin{equation} \label{eqKuram}
\pat u +u \pax u + \epsilon \Lambda^{1+\delta}u=\Lambda^\gamma u, \quad 0\leq\gamma\leq1+\delta, \quad \epsilon,\delta>0.
\end{equation}
It provides a simple model for the stabilization of a Hadamard instability, with growth rate proportional to the absolute value of the wavenumber,
by higher-order viscous diffusion. The Kelvin-Helmholtz instability for the Euler equations is an instability of this type. Equation \eqref{eqKuram} was studied by Granero-Belinch\'on \& Hunter \cite{GH}. These authors showed the chaotic character of the dynamics and studied analytically some properties of the attractor and the solutions. 

Next, let us recall that extensive studies were performed for  the so called $\delta-$model
\begin{equation} \label{eqdelta}
\pat u +(1-\delta)H u \pax u +\delta \pax (uHu) + \nu \Lambda^{\alpha}u=0, \quad 0\leq\delta\leq1.
\end{equation}
Let us for a moment focus on the the inviscid case $\nu=0$.
This equation appears as a one-dimensional model of the inviscid surface quasi-geostrophic (SQG) equation. From the viewpoint of applications, the SQG equation models the dynamics of hot and cold air in the atmosphere; from a purely mathematical perspective, it is widely studied as a two dimensional model of the three dimensional Euler equations in its vorticity formulation. Let us note that there are several singularity formation results available. The singularity in the cases $0<\delta<1/3$ and $\delta=1$ was proved by Morlet \cite{Morlet}. In the range $0<\delta\leq1$ the finite time singularity was proved by Chae, C\'ordoba, C\'ordoba \& Fontelos \cite{CCCF}. Finally, the singularity in the case $\delta=0$ was proved by C\'ordoba, C\'ordoba and Fontelos \cite{gazolaz2005formation} and by Silvestre \& Vicol \cite{silvestre2014transport}. The case $\delta=-1$ and $\alpha>1$ is similar to the nonlocal Kuramoto-Sivashinsky equation \eqref{eqKuram} and it has been studied by Li \& Rodrigo in \cite{LiRodrigo}. On the other hand, Bae \& Granero-Belinch\'on \cite{bae2015global} proved the global existence of weak solution in the case $1>\delta\geq 0.5$.

The two extreme cases $\delta=0$ and $\delta=1$, both in the inviscid and viscous case, deserve special further attention. 

The case $\delta=0$ reads
\begin{equation} \label{eqdelta0}
\pat u +H u \pax u + \nu\Lambda^{\alpha}u=0.
\end{equation}
This is a Burgers-like equation with a nonlocal velocity. It is contained in the scale of equations
\begin{equation} \label{eqdelta0b}
\pat u +\Lambda^sH u \pax u + \nu\Lambda^{\alpha}u=0,\quad s\in(-1,1).
\end{equation}
The family of equations \eqref{eqdelta0b} was proposed as a model of the two-dimensional $\alpha-$patches by Dong \& Li \cite{dong2014one}. 

As we have outlined, in the inviscid case, equation \eqref{eqdelta0} develops finite time singularities. On the other hand, one can prove that if $\alpha>1$ there exist global classical solutions. The global existence for arbitrary initial data in the critical case $\alpha=1$ was proved by Dong \cite{HDong}, see also  Kiselev \cite{Kiselev}. The singularity formation in the viscous case has been proved in the  range $0\leq \alpha< 0.25$ by Li \& Rodrigo \cite{LiRodrigo2}. The range $0.25\leq \alpha<1$ remains a major open problem. In this regard, it is worth mentioning that Silvestre \& Vicol \cite{silvestre2014transport} conjecture that solutions of \eqref{eqdelta0} with $0.25\leq \alpha<1$ which arise as limits from vanishing viscosity approximations are bounded in the H\"{o}lder class $1/2$. 

The case $\delta=1$ of \eqref{eqdelta}, or, equivalently, equation \eqref{eqdislocation} in the inviscid case, has been studied by many authors. For instance, global classical solutions for initial data satisfying a strict positivity condition and finite time singularities otherwise are proved by Castro and C\'ordoba in \cite{CC} (see also \cite{CC2}). The existence of global weak solutions is proved by Carrillo, Ferreira \& Precioso \cite{Carrillo} or Bae \& Granero-Belinch\'on \cite{bae2015global}. The authors studied in \cite{BG} the asymptotic behavior. In particular, we obtained the following inequality
$$
\frac{d}{dt}\left(\mathcal{F}(u)-\frac{1}{2\min_{x}u_0(x)}\mathcal{I}(u)\right)\geq0,
$$
where the entropy and the Fisher information are defined as
$$
\mathcal{F}(u)=\int_\TT u\log(u)-udx,
$$
$$
\mathcal{I}(u)=\int_\TT \left|\Lambda^{0.5} u\right|^2dx,
$$
respectively, and $u_0$ such that $\langle u_0\rangle=1$ denotes the initial data for equation \eqref{eqdislocation}. As a consequence, we have proved the following \emph{nonlocal logarithmic Sobolev} inequality
\begin{equation}\label{logsobnonlocal}
\int_\TT f(x)\log(f(x))dx\leq 2\pi +\frac{1}{2\min_{x}f(x)}\int_\TT \left|\Lambda^{0.5} f(x)\right|^2dx.
\end{equation}
Inequality \eqref{logsobnonlocal} remains valid for every function $f\in L^1\cap \dot{H}^{0.5}$, $f>0$. Consequently, for strictly positive and bounded below data we have obtained exponential decay of $\mathcal{F}(u(t))$.

Equation \eqref{eqdislocation} is a particular case of the aggregation equation
\begin{equation}\label{eqdislocationagg}
\pat u  =  \pax(-\mu(u)Hu+\chi u\pax (\pax^2)^{-1} (u-\langle u \rangle) +ru(1-u),
\end{equation}
studied by Granero-Belinch\'on \& Orive-Illera \cite{GO} and the authors in \cite{BG}. In turn, equation \eqref{eqdislocationagg} is akin to the classical (parabolic-elliptic) Keller-Segel equation in one spatial dimension:
\begin{equation}\label{eqKS1D}
\pat u  = -\Lambda^\alpha u +\pax(u\pax (\pax^2)^{-1} (u-\langle u \rangle)).
\end{equation}
Since Keller-Segel systems are one of our main interests, let us concentrate on it for a moment. Equation \eqref{eqKS1D} is the one dimensional model of the behavior of microorganisms (with density $u$), where their motility follows $\Lambda^\alpha$, instead the (more classical) $-\pax^2$. This choice for the diffusive operator is supported by a strong evidence, both theoretical and experimental, that feeding strategies based on L\'evy process are used by organisms in certain situations. For instance, the interested reader can refer to Lewandowsky, White \& Schuster \cite{Lew_nencki} for amoebas, Klafter, Lewandowsky \& White \cite{Klaf90} and Bartumeus, Peters, Pueyo, Marras{\'e} \& Catalan \cite{Bart03} for microzooplancton, Shlesinger \& Klafter \cite{Shl86} for flying ants, Cole \cite{Cole} for fruit flies, Atkinson, Rhodes, MacDonald \& Anderson \cite{Atk} for jackals and Raichlen, Wood, Gordon, Mabulla, Marlowe, \& Pontzer for the Hadza tribe \cite{hadza}. 

Equation \eqref{eqKS1D} was first studied by Escudero \cite{escudero2006fractional}. He proved the global existence of solution in the case $1<\alpha\leq 2$. This result was later improved by Bournaveas \& Calvez \cite{bournaveas2010one}, where the authors proved finite time singularities for the case $0<\alpha<1$ and the existence of $K>0$ such that for the case $\alpha=1$ and the initial data satisfying the smallness restriction
$$
\|u_0\|_{L^1}\leq K,
$$
there exists a global smooth solution. Ascasibar, Granero-Belinch\'on \& Moreno \cite{AGM} obtained the global existence in the case $\alpha>1$ and $\alpha=1$ under the \emph{relaxed smallness condition} restriction
$$
\|u_0\|_{L^1}\leq \frac{1}{2\pi}.
$$
Under this restriction, the solution $u(t)$ verifies
$$
\|u(t)-\langle u_0 \rangle\|_{L^\infty}\rightarrow 0.
$$
All these results where improved by the authors in \cite{BG3}, where we showed the global existence and smoothness in the case $\alpha=1$  and arbitrarily large initial data. Furthermore, we also showed that if the initial data $u_0$ is such that
$$
\|u_0\|_{L^1}< 2\pi,
$$
then the solution $u(t)$ verifies the asymptotic behaviour
$$
\|u(t)-\langle u_0 \rangle\|_{L^2}\leq \Sigma e^{-\sigma t},
$$
for certain $\Sigma,\sigma>0$.

 Let us mention also here the works by Biler \& Karch \cite{BilKar10}, Biler \& Wu \cite{BilerWu} and  Wu \& Zheng \cite{WuZheng}, where nonlocal Keller-Segel-type models with very similar in spirit operators $B(u)$ are considered.

Finally, let us mention in our review of one-dimensional models the equation
\begin{equation}\label{modelodiegoderiv}
\partial_t u = -\Lambda u + \pax\left(\frac{u^2Hu}{1+u^2}\right).
\end{equation}
It arises as a model of the interface in two-phase, porous flow. Equation \eqref{modelodiegoderiv} was considered in \cite{c-g-o08} by C\'ordoba, Gancedo \& Orive-Illera, where the authors studied the local existence and qualitative behavior of the solutions. Granero-Belinch\'on, Navarro \& Ortega \cite{GNO} proved the global existence of weak solutions.

\section{Statement of results}\label{sec:3}
We start with our definition of weak solution

\begin{defi}\label{def:weaksol}$u(x,t)\in L^2(0,T;L^2)$ is a global weak solution of \eqref{eq1}, \eqref{eqB} (or of \eqref{eq1}, \eqref{eqB2}) with logistic forcing \eqref{eqf} if for all $T>0$ and $\phi\in \mathcal{D}([-1,T)\times \TT)$ we have
$$
\int_0^T\int_\TT u(-\pat \phi+\Lambda^\alpha\phi)+uB(u)\pax\phi-f(u)\phi dxds = \int_\TT u_0\phi(0)dx.
$$
\end{defi}
Let us observe that the bigger the ratio $\chi / r$ is in \eqref{eq1}, the stronger is the destabilization coming from the term $\chi\pax(u B(u))$, compared with the logistic damping $f(u) = ru(1-u)$. Nevertheless, we will be able to show that for \emph{any} ratio $\chi / r$, there is certain region in the supercritical regime, where the stabilizing effect prevails. 
 
Our first result is the global existence of weak solutions with certain higher regularity. Define
$$
s=\min \left( \frac{r}{\chi-r}, 1 \right)
$$
and $s^-$ as any number strictly smaller than $s$.

\begin{teo}\label{th:globalweak}
Take $\alpha,\chi,\beta,r>0$. Let $u_0\geq0$, $u_0\in L^{2}$ be the initial data. \\
(i) Case $r < \chi/2$. \\
Assume that 
\begin{equation}\label{eqhard}
\alpha>1-s.
\end{equation}
Then, there exists a global in time weak solution
$$
u\in L^{2+s^-} (0,T; L^{2+s^-})
$$ for \eqref{eq1},\eqref{eqB} (or \eqref{eq1}, \eqref{eqB2}) with \eqref{eqf} (in the sense of Definition \ref{def:weaksol}). This solution verifies for any $T <\infty$
$$
u\in L^\infty(0,T;L^{1+s})\cap L^2(0,T;L^{2}\cap W^{\alpha/2-\delta_1,1})\cap L^{2+2s} (0,T;W^{\alpha/(2+2s)-\delta,1+s}),
$$
where $0<\delta<\frac{\alpha}{2+2s}$ and  $0<\delta_1<\frac{\alpha}{2}$ are arbitrarily small numbers.\\
At the cost of decreasing $s$ to $s^-$ above, one can relax condition \eqref{eqhard} and assume instead
\begin{equation}\label{eqsoft}
\alpha > 2 \frac{1-s-s^2}{s+2}.
\end{equation} 
(ii) Case $r \geq \chi/2$. \\
In addition to the previous case, we have here for any $T <\infty$
$$
u\in L^2(0,T;H^{\alpha/2}).
$$
\end{teo}

In other words, for arbitrarily large ratio $\chi / r$ we have global existence of weak solutions in the supercritical regime $\alpha>1-s$. Moreover, once the logistic term is strong enough (case $r>\chi/2$) every diffusion $\alpha>0$ is enough to provide us with global solutions. This result is particularly interesting for the case $0<r\ll1$. 

\begin{remark}
Notice that in the case $r=0$ and every $\alpha>0$, global existence of (at least one) weak solution in the measure space $\dot{H}^{-1}$ can be obtained (see \cite{BG3}).
\end{remark}

Our second result establishes the global existence of smooth solution for a range of diffusions $\max(1-r/\chi, 0)<\alpha\leq 2$ and sufficiently smooth initial data. In particular, notice that for any $r>0$ there are diffusions with order less than one (supercritical range) such that the solution is global and smooth. 

\begin{teo}\label{th:globalstrong}
Take $r,\chi, \alpha, \beta>0$ and $u_0\geq0$, $u_0\in H^{3}$. Then,
if 
$$
1-\frac{r}{\chi}<\alpha\leq 2,
$$
there exists the unique global in time solution for \eqref{eq1},\eqref{eqB} (or for \eqref{eq1}, \eqref{eqB2}) with \eqref{eqf}, verifying for any $T < \infty$
$$
u\in C([0,T],H^3).
$$
Furthermore, $u(t)$  is real analytic for any $t \in (0, T]$ and verifies
$$
\sup_{0\leq t <\infty}\|u(t)\|_{L^1}\leq \max\{2\pi,\|u_0\|_{L^1}\}.
$$
\end{teo}

In particular, we obtain the following corollary:
\begin{corol}
Given $r\geq\chi>0$, the system \eqref{eq1}, \eqref{eqB} (or \eqref{eq1}, \eqref{eqB2}) with \eqref{eqf} posesses a global in time, classical solution for every initial data $u_0 \in H^{3}$. 
\end{corol}

The idea of the proof of Theorem \ref{th:globalstrong} is to use \emph{pointwise} bounds on the $\Lambda^\alpha u$ term that control the reaction term and prevent blow-up. This method has been developed by C\'ordoba \& C\'ordoba \cite{cordoba2003pointwise, cor2} and Constantin \& Vicol \cite{constantin2012nonlinear}. It has proved itself useful in many different equations: see \cite{cor2,constantin2013long} for the critical surface quasi-geostrophic equation, \cite{c-g09, c-g-o08, ccfgl, CGO, ccgs-10, G} for its application to the Muskat problem and \cite{AGM, GO, BG} for its application to the aggregation equations of Keller-Segel type. 

Notice also that the solutions are global but, as far as we know, they may be growing in $\|u (t) \|_{L^\infty}$ without bound. In other words, the logistic term is not strong enough to provide global, uniform estimates. In particular, the estimates for $\|u\|_{L^\infty}$ depend on the estimate for the diffusion term. In this context, let us provide a result concerning the large-time behaviour:
\begin{prop}\label{th:largetime}
Let $r,\chi,\beta>0$, $\alpha>1$ and $u_0\geq0$, $u_0\in H^{3}$. Then,
the global in time solution for \eqref{eq1}, \eqref{eqB} (or \eqref{eq1}, \eqref{eqB2}) with \eqref{eqf} verifies
$$
\sup_{0\leq t <\infty}\|u(t)\|_{L^\infty}\leq C(\alpha,r,\chi, \|u_0\|_{L^\infty},\|u_0\|_{L^1}).
$$
For $\alpha=1$ we have 
$$
\sup_{0\leq t <\infty}\|u(t)\|_{L^\infty}\leq C(r,\chi, \|u_0\|_{L^\infty},\|u_0\|_{L^1}),
$$
provided $\chi < r+\frac{1}{2\pi \max(\| u_0\|_{L^1}, 2 \pi)}$.
\end{prop}
\begin{corol}
For $\alpha=1$ 
the global in time solution for \eqref{eq1}, \eqref{eqB} (or \eqref{eq1}, \eqref{eqB2}) with \eqref{eqf} verifies
$$
\sup_{0\leq t <\infty}\|u(t)\|_{L^\infty}\leq C(\alpha,r,\chi,  \|u_0\|_{L^\infty},\|u_0\|_{L^1}).
$$
if one of the following conditions is satisfied
\begin{itemize}
\item $\chi \le r$ and the initial data may be arbitrarily large.
\item $\|u(t)\|_{L^1} \le 2 \pi$ and $\chi < r+\frac{1}{4\pi^2}$. In particular, for  $\chi \le \frac{1}{4\pi^2}$ any $r>0$ is admissible.
\end{itemize}
\end{corol}

\section{Preliminary results}\label{sec:4}
\subsection{Elliptic estimates for $v$}
We need certain elliptic estimates of $v$, defined in \eqref{eqvB} and \eqref{eqvB2}:
\begin{lem}\label{lemmaboundsv} For sufficiently smooth, nonnegative function $u$ and $v$ given by  \eqref{eqvB} one has
\begin{enumerate}
\item $
\min_x v(x)\geq0, $ in particular $\Lambda^\beta v(x)\leq u(x)$,
\item 
$ \|v (t) \|_{L^\infty}\leq \|u(t)\|_{L^\infty}, $
\item 
$ \|\Lambda^\beta v(t)\|_{L^\infty }\leq 4\|u(t)\|_{L^\infty }, $
\item 
$
\|\pax \Lambda^\beta v (t) \|_{L^\infty}\leq 4\|\pax u (t) \|_{L^\infty},
$
\item 
$
\frac{1}{2}\|\Lambda^{\beta+s} v (t) \|_{L^2}^2+\|\Lambda^{\beta/2+s}v (t) \|_{L^2}^2\leq \frac{1}{2}\|\Lambda^s u (t) \|_{L^2}^2,\quad \forall\,s\geq0.
$
\end{enumerate}
\end{lem}
\begin{proof}
Recall that  $v$ given by \eqref{eqvB} means that it  is a solution to
\begin{equation}\label{eqv}
v+\Lambda^\beta v=u.
\end{equation}
\textbf{Part (1)} Let us evaluate this equation at $x_t$ such that $\min_{x}v(x,t)=v(x_t,t)$. Due to the definition of $x_t$, we have $v(x_t)-v(x_t-\eta)\leq0$, and, consequently, using \eqref{Lkernelalpha}, $\Lambda^\beta v(x_t)\leq0$. Using this fact we get
\begin{equation}\label{boundlapbeta2c}
v(x_t)\geq\Lambda^\beta v(x_t) + v(x_t) \geq \min_{x}u(x,t)\geq0.
\end{equation}
The non negativity of $u$ starting from nonnegative data follows from a similar technique of tracking minimum, compare \cite{BG}. \\
\textbf{Part (2)} In the same way, if we evaluate at $\tilde{x}_t$ such that $\max_{x}v(x,t)=v(\tilde{x}_t,t)$, we have $\Lambda^\beta v(\tilde{x}_t)\geq0$ and we get the bound
$$
v(\tilde{x}_t)\leq v(\tilde{x}_t)+\Lambda^\beta v(\tilde{x}_t) \leq \max_{x}u(x,t).
$$
\textbf{Part (3) and (4).} Now we evaluate the equation \eqref{eqv} at $X_t$ such that $\max_{x}\Lambda^\beta v(x,t)=\Lambda^\beta v(X_t,t)\geq0$. Due to part (2), we have
\begin{equation}\label{boundlapbeta2}
\Lambda^\beta v(X_t)= u(X_t) - v(X_t) \leq 2\|u(t)\|_{L^\infty}.
\end{equation}
Similarly, for $\tilde{X}_t$ such that $\min_{x}\Lambda^\beta v(x,t)=\Lambda^\beta v(\tilde{X}_t,t)\leq0$ we have
\begin{equation}\label{boundlapbeta1}
-\Lambda^\beta v(\tilde{X}_t)=-u(\tilde{X}_t) +  v(\tilde{X}_t) \leq 2\|u(t)\|_{L^\infty}.
\end{equation}
Collecting estimates \eqref{boundlapbeta2} and \eqref{boundlapbeta1}, we obtain $(3)$. In the same way, we can prove $(4)$.

\textbf{Part (5).} It follows from \eqref{eqv}.
\end{proof}

In turn, the case where $v$ is defined by \eqref{eqvB2} is studied in the following lemma.
\begin{lem}\label{lemmaboundsv2}  For sufficiently smooth, nonnegative function $u$ and $v$ given by  \eqref{eqvB} one has
\begin{enumerate}
\item 
$
\|\Lambda^\beta v(t)\|_{L^\infty }\leq \|u(t)\|_{L^\infty },
$
\item 
$
\|\pax \Lambda^\beta v(t)\|_{L^\infty}\leq \|\pax u(t)\|_{L^\infty},
$
\item 
$
\frac{1}{2}\|\Lambda^{\beta+s} v(t)\|_{L^2}^2\leq \frac{1}{2}\|\Lambda^s u(t)\|_{L^2}^2,\quad \forall\,s\geq0.
$
\end{enumerate}
\end{lem}
\begin{proof}
In this case, $v$ is a solution to
\begin{equation}\label{eqv2}
\Lambda^\beta v=u-\langle u\rangle.
\end{equation}
Now, the proof is a simpler version of respective part of proof of the previous lemma.
\end{proof}

\subsection{Local existence}
We start with the local existence result.
\begin{lem}\label{localexistence} Let $u_0\in H^3$ be a non-negative initial data. Then, if $0\leq \alpha\leq 2$, there exists a time $T^*(u_0)>0$ such that there exists a non-negative solution 
$$
u(t)\in C([0,T^*(u_0)],H^3)
$$
to the equation \eqref{eq1}, \eqref{eqB} (or to  \eqref{eq1}, \eqref{eqB2}) with \eqref{eqf}. Moreover, if for a given $T$ the solution verifies the following bound
$$
\int_0^T \|u(s)\|_{L^\infty}ds<\infty,
$$
then the solution may be extended up to time $T+\delta$ for small enough $0<\delta$. Furthermore, under the restriction $1\leq \alpha\leq2$, the solution becomes real analytic.
\end{lem}
Once we have Lemmas \ref{lemmaboundsv} and \ref{lemmaboundsv2}, the proof of this result is similar to the proofs in \cite{AGM, GO, BG, BG2}. For brevity, we omit the standard proof.

\subsection{Auxiliary energy estimates}
In the next lemma we collect some bounds for Lebesgue and Sobolev norms of the solution.

\begin{lem}\label{Lemmanorms} Let $u_0\in H^3$ be a non-negative initial data. Take $s\in\RR^+$ and $\chi>r>0$ such that
$$
0<s\leq \frac{r}{\chi-r}.
$$
Given an initial data $u_0$ for equation \eqref{eq1}, \eqref{eqB} (or for \eqref{eq1},  \eqref{eqB2}) with \eqref{eqf}, let us define 
\begin{equation}\label{calN}
\mathcal{N}=\max\{\|u_0\|_{L^1},2\pi\}.
\end{equation}
Then, the non-negative solution $u(x,t)$ verifies for all times $0\leq t\leq T< T_{max}$ the following bounds
\begin{enumerate}
\item
$
\|u(t)\|_{L^1}\leq \mathcal{N}, 
$ \vskip 1mm
\item
$
\int_0^t\|u(s)\|^2_{L^2 }ds\leq \mathcal{N}t+2\mathcal{N},
$ \vskip 1mm
\item 
$
\|u(t) \|_{L^{s+1}} \le e^{r t}\| u_0 \|_{L^{s+1}},
$  \vskip 1mm
\item $
 \left(  r (s+1) - \chi s \right)  \int_0^t \|u(s) \|^{s+2}_{L^{s+2}}  ds \le e^{r(s+1) t} \| u_0 \|^{s+1}_{L^{s+1}},
$  \vskip 1mm

\item
$
\int_0^t\|u(s)\|_{\dot{W}^{\alpha/(2+2s)-\delta,1+s}}^{2+2s}\leq C(\alpha,s,\delta,\|u_0\|_{L^{1+s}},T)
$ \\
for arbitrary $0<\delta<\alpha/(2+2s)$,  \vskip 1mm
\item $
\int_0^t\|u(s)\|_{\dot{W}^{\alpha/2-\delta,1}}^2 ds\leq C(\alpha,\delta,\mathcal{N}, \|u_0\|_{L^{1+s}}, T)
$ \\
for arbitrary $0<\delta<\alpha/2$.
\end{enumerate}
\end{lem}
\begin{proof}
%We consider $B$ given by \eqref{eqB}. The case of $B$ given by \eqref{eqB2} is similar. 

\textbf{Part (1) and (2).} The proof of this result is analogous to the proof in \cite[Lemmas 1, 2]{BG2}. 

\textbf{Part (3).} Testing \eqref{eq1} with $u^s$ we get after integrations by parts
\begin{equation}\label{eq:ener}
\frac{1}{s+1} \frac{d}{dt} \int_\T u^{s+1}dx +     \int_\T u^s  \Lambda^\alpha u \, dx  = \int_\T  \frac{\chi s}{s+1}  (u^{s+1})  \Lambda^\beta v  +ru^{s+1} -ru^{s+2}dx.
\end{equation}
Lemma \ref{lemaentropy2} implies
$$
0\leq  \int_\TT u^s(x)\Lambda^{\alpha}u(x)dx.
$$
This and $\Lambda^\beta v \le u$ (valid due to definition  \eqref{eqv} and Lemma  \ref{lemmaboundsv}  or due to definition  \eqref{eqv2}) gives

\begin{equation}\label{eq:gr2}
 \frac{d}{dt} \int_\T u^{s+1}  + \left( r (s+1) - \chi s \right)  \int_\T u^{s+2} \le r (s+1) \int_\T  u^{s+1}. %-\lambda \chi s \int_\T  u^{s+1}  v .
\end{equation}
%where we have used Lemma \ref{lemmaboundsv} and \eqref{boundlapbeta2c} (for the case $v$ defined in \eqref{eqv}) and \eqref{eqv2} for the case $B(u)$ defined in \eqref{eqB2}. 
Therefore, for $s$ such that $(\chi - r)s \le r$ we get
\begin{equation}\label{eq:gr}
\|u(t) \|_{L^{s+1}} \le e^{r t}\| u_0 \|_{L^{s+1}}.
\end{equation}

\textbf{Part (4)} Integrating \eqref{eq:gr2} in time we obtain that for any $t \in [0, T]$,
%\begin{equation}\label{eq:lp}
\[
 \left(  r (s+1) - \chi s \right)  \int_0^t \int_\T u^{s+2} (x, s) dx ds \le [r(s+1) e^{r(s+1)t } +1] \int_\T u_0^{s+1} (x) dx,
 \]
%\end{equation}
where the latter comes from \eqref{eq:gr}. 

\textbf{Part (5)} Integrating \eqref{eq:ener} in time and using \eqref{eq:gr} together with Lemma \ref{lemaentropy2}, we obtain
$$
\int_0^T\|u\|_{\dot{W}^{\alpha/(2+2s)-\delta,1+s}}^{2+2s}\leq C(\alpha,s,\delta,u_0,T).
$$

\textbf{Part (6).} We define the functional
$$
\mathcal{F}=\int_\TT u\log(u)-u +1 dx.
$$

We have
\begin{align*}
\frac{d}{dt}\mathcal{F}&=\int_\TT \pat u\log(u)dx\\
&\leq-\int_\TT \Lambda^\alpha u\log(u) dx+\chi\int_\TT \Lambda^\beta v u dx\\
&\leq -\int_\TT \Lambda^\alpha u\log(u) dx+\chi\|u\|_{L^2}^2.
\end{align*}
Consequently, we obtain
$$
\mathcal{F}(u(t))+\int_0^t\int_\TT \Lambda^\alpha u\log(u) dx\leq \mathcal{F}(u(0))+\mathcal{N}t+2\mathcal{N}.
$$
Using
$$
\mathcal{F}(u(0))\leq \int_{\{u(x,0)\geq1\}}u(x,0)^{1+s}dx +2 \pi\leq \|u(0)\|_{L^{1+s}}+2\pi,
$$
\begin{equation}\label{uniformboundentropy}
\mathcal{F}(u(t))+\int_0^t\int_\TT \Lambda^\alpha (u+1)\log(u+1) \leq \|u_0\|_{L^{1+s}}+2\pi+ 2 \chi \mathcal{N} (T+2).
\end{equation}
We conclude by Lemma \ref{lemaentropy}.
\end{proof}
\section{Proof of Theorem \ref{th:globalweak}}\label{sec:5}
For the sake of brevity, we are going to prove the case $B$ defined in \eqref{eqB}. The remaining case $B$ given by \eqref{eqB2} requires only minor modifications. We only consider the case $\chi > r$, being the other case much easier.

\subsection{Approximate problems}
Let $T>0$ be a fixed parameter. For $\epsilon \in (0,1)$ we define the following \emph{approximate} problem 
\begin{equation}\label{eqapprox}
\pat u_\epsilon=-\epsilon\Lambda ^{1.75} u_\epsilon-\Lambda ^\alpha u_\epsilon+ \chi \pax(u_\epsilon B(u_\epsilon))+f(u_\epsilon),
\end{equation}
where $B$ is given by \eqref{eqB} (or by \eqref{eqB2}) and $f$ is given by \eqref{eqf}. The function $u$ verifies the initial condition
$$
u_\epsilon(0,x)=\jeps u_0(x)\geq0;
$$
where $\jeps$ is the heat kernel at time $\epsilon$. Along the lines of Lemma \ref{localexistence} we have local existence of solution. Let us define
$$
u_\epsilon(x_t,t)=\max_{x}u_\epsilon(x,t),
$$
and
$$
r_0=\frac{2\mathcal{N}}{u_\epsilon(x_t)}.
$$
We have two possibilities, namely $r_0\geq\pi$ or $r_0\leq \pi$. Assume first that $r_0\geq\pi$. Then we have the bound
\begin{equation}\label{eqclear}
\frac{2\mathcal{N}}{\pi}\geq \|u_\epsilon(t)\|_{L^\infty}.
\end{equation}
On the other hand, if $r_0\leq \pi$, we can apply Lemma \ref{lemaaux3} with $p=1, \gamma_1 = \mathcal{N}$ and we get
$$
\Lambda ^{1.75} u_\epsilon(x_t)\geq C(\mathcal{N})\|u_\epsilon(t)\|_{L^\infty}^{2.75}.
$$
Using pointwise methods as in \cite{AGM,BG,GO}, we have
$$
\frac{d}{dt}\|u_\epsilon(t)\|_{L^\infty}\leq -\epsilon C(\mathcal{N})\|u_\epsilon(t)\|_{L^\infty}^{2.75}+\chi\|u_\epsilon(t)\|_{L^\infty}^2+r\|u_\epsilon(t)\|_{L^\infty}(1-\|u_\epsilon(t)\|_{L^\infty}).
$$
As a consequence, we obtain a uniform-in-time bound
\begin{equation}\label{eqclear2}
\|u_\epsilon(t)\|_{L^\infty}\leq C(\epsilon,\chi,r,\mathcal{N}).
\end{equation}
Collecting both inequalities \eqref{eqclear} and \eqref{eqclear2}, we have the bound
$$
\|u_\epsilon(t)\|_{L^\infty}\leq C(\epsilon,\chi,r,\mathcal{N})+\frac{2\mathcal{N}}{\pi},
$$
regardless of the value of $r_0$.

The continuation criterion implies that $u_\epsilon$ is a global-in-time, smooth solution (recall that $T$ is arbitrary, but fixed).

\subsection{Uniform estimates}

\textbf{Part 1. } We will use in what follows thesis of Lemma \ref{Lemmanorms} for $u_\epsilon$ in place of $u$, since it holds for  $u_\epsilon$ by the same token as for $u$.

Let us consider first the case $r \geq \chi/2$. Using this condition, we take 
$$
s=1\leq\frac{r}{\chi-r}.
$$

Lemma \ref{Lemmanorms} gives the following uniform bounds for all $t\leq T$:
\begin{equation}\label{equnif1}
\sup_{0\leq t\leq T}\|u_\epsilon(t)\|_{L^{1}}\leq \mathcal{N},
\end{equation}

\begin{equation}\label{equnif2}
\int_0^t\|u_\epsilon(s)\|_{\dot{W}^{\alpha/2-\delta,1}}^2 ds\leq C(\alpha,\delta,\mathcal{N}, \|u_0\|_{L^{2}}, T),
\end{equation}
for arbitrary $0<\delta<\alpha/2$.

\begin{equation}\label{equnif3}
\sup_{0\leq t\leq T}\|u_\epsilon(t)\|_{L^{1+s}}\leq C(r,T,\|u_0\|_{L^{1+s}}),
\end{equation}

\begin{equation}\label{equnif4}
\int_0^t\|u_\epsilon(s)\|^{2}_{L^{2}}ds \leq C(r,T,\|u_0\|_{L^{1}},\chi).
\end{equation}

\begin{equation}\label{equnif41}
 \left(  r (s+1) - \chi s \right)  \int_0^t \|u(s) \|^{s+2}_{L^{s+2}}  ds \le C(r,T,\|u_0\|_{L^{1+s}}),
\end{equation}

\begin{equation}\label{equnif42}
\int_0^t\|u(s)\|_{\dot{W}^{\alpha/(2+2s)-\delta,1+s}}^{2+2s}\leq C(\alpha,s,\delta,\|u_0\|_{L^{1+s}},T)
\end{equation}
for arbitrary $0<\delta<\alpha/(2+2s)$,  \vskip 1mm

Furthermore, by \eqref{eq:ener} and following the proof of Lemma \ref{Lemmanorms}, we also have
\begin{equation}\label{equnif4a}
\int_0^t\|u_\epsilon (s)\|_{\dot{H}^{\alpha/2}}^2ds\leq C(r,T,\|u_0\|_{L^{2}},\chi).
\end{equation}

\textbf{Part 2.} Now, let us consider the case $ r < \chi/2$. Using Lemma \ref{Lemmanorms} and taking 
\begin{equation}\label{equnif4ch}
s=\frac{r}{\chi-r} \quad (<1),
\end{equation}
the uniform bounds \eqref{equnif1}-\eqref{equnif4} hold. We do not have \eqref{equnif4a}, but Lemma \ref{lemaentropy2} and \eqref{equnif3} imply
$$
\int_0^t\|u_\epsilon(s)\|_{\dot{W}^{\alpha/(2+2s)-\delta,1+s}}^{2+2s} ds \leq C(\alpha,s,\delta,\|u_0\|_{L^{1+s}},T)
$$
for arbitrary $0<\delta<\alpha/(2+2s)$.

\subsection{Convergence}
Due to the uniform bound \eqref{equnif4} we have
$$
\pat u_\epsilon\in L^2(0,T;H^{-2})
$$
$\epsilon$-uniformly bounded.

\textbf{Part 1.} Now we first consider the case $r\geq \chi/2$. We take
$$
X_{1}=H^{-2},\,X=L^2,\,X_0=H^{\alpha/2},
$$
Then, due to the Aubin-Lions Theorem, we have that
$$
Y=\{h: \, h\in L^2(0,T;X_0), \; \pat h \in L^2(0,T;X_1)\}
$$
is compactly embedded in 
$$
L^2(0,T;X).
$$
Hence we can extract a subsequence (denoted again by $u_\epsilon$) such that
$$
u_\epsilon\rightharpoonup u \text{ in }L^2(0,T;H^{\alpha/2}),
$$
$$
u_\epsilon\rightarrow u \text{ in }L^2(0,T;L^2).
$$
This strong compactness is enough to pass to the limit in the weak formulation. As a consequence, $u$ is a global weak solution (in the sense of Definition \ref{def:weaksol}) of the original system \eqref{eq1}, \eqref{eqB} (or \eqref{eq1}, \eqref{eqB2}) with \eqref{eqf}. Its regularity follows from l.w.s.c. of norms and uniform bounds \eqref{equnif2} - \eqref{equnif4a}.

\textbf{Part 2.} Now we explain how to adapt the proof for the second case  $r < \chi/2$. 
Let us first use the assumption $\alpha  >1 - s$. It implies
$$
\frac{1}{2}>\frac{1}{1+s}\left(1-\frac{\alpha}{2}\right),
$$
consequently, for $\delta>0$ sufficiently small, 
$$
W^{\alpha/(2+2s)-\delta,1+s}\subset\subset L^2.
$$
We take
$$
X_{1}=H^{-2},\,X=L^2,\,X_0=W^{\alpha/(2+2s)-\delta,1+s},
$$
and we apply Aubin-Lions's Theorem again. The rest of the proof remains unchanged.

Next, let us focus on the assumption  $ \alpha > 2 \frac{1-s-s^2}{s+2}$, where, abusing notation, we have redefined
$$
0<s<\frac{r}{\chi-r}.
$$ It implies
\[W^{\alpha/(2+2s)-\delta,1+s}\subset\subset L^\frac{s+2}{s+1}.\]
We choose
$$
X_{1}=H^{-2},\,X=L^\frac{s+2}{s+1},\,X_0=W^{\alpha/(2+2s)-\delta,1+s},
$$
the Aubin-Lions Theorem implies now that one can extract a subsequence (denoted again by $u_\epsilon$) such that
$$
u_\epsilon\rightarrow u \text{ in }L^{2+2s}(0,T;L^\frac{s+2}{s+1}).
$$
For the limit passage in the most troublesome part $-r u^2_\epsilon $ we want to use additionally (4) of Lemma \ref{Lemmanorms} for $u_\epsilon$, i.e.
\[ \left(  r (s+1) - \chi s \right)  \int_0^t \|u_\epsilon(s) \|^{s+2}_{L^{s+2}}  ds \le e^{r(s+1) t} \| u_0 \|^{s+1}_{L^{s+1}}. \]
For its left-hand-side to be meaningful, we need to choose now in 
\eqref{equnif4ch}
the sharp inequality. Consequently, we have
\[r \int_0^T \int_\TT  | (u^2_\epsilon - u^2) \phi| \le r |\phi|_{L^\infty_t L_x^\infty}  \| u_\epsilon - u\|_{L^2_tL^\frac{s+2}{s+1}_x}  \| u_\epsilon +u\|_{L^2_tL_x^{s+2}}   \]
and the right-hand-side vanishes as $\epsilon \to 0$ thanks to strong convergence in $L^2(0,T;L^\frac{s+2}{s+1})$ and boundedness in $L^{s+2}( 0, T; L^{s+2})$.

\section{Proof of Theorem \ref{th:globalstrong}}\label{sec:6}
For the sake of brevity, we are going to prove the case of $B$ defined by \eqref{eqB}. For the other one ($B$ given by \eqref{eqB2}), the proof can be easily adapted. 

For $\alpha=2$,  the proof follows from standard energy methods and Sobolev inequality. Hence, let us focus on the case $2>\alpha$. First, let us fix an arbitrary $T < \infty$. The proof is based on obtaining the bound
$$
\sup_{0\leq t\leq T}\|u(t)\|_{L^\infty}\leq C(\alpha,r,\chi,u_0, T),
$$
where the constant $C(\alpha,r,\chi,u_0, T)$ is finite for finite $T$.

A computation gives (see \cite{cordoba2003pointwise, cor2, constantin2013long, constantin2012nonlinear})
\begin{equation}\label{eqpointwise}
f \Lambda^\alpha f = \frac{1}{2} \Lambda^\alpha (f^2) + \frac{1}{2}  I (f),
\end{equation}
with 
\[
I (f)  := c_{\alpha}\;\text{P.V.} \int_{\mathbb{T}} \frac{(f(x)-f(y))^2}{|x-y|^{1+\alpha}}dy+c_{\alpha}\sum_{k\in\mathbb{Z}\setminus\{0\}} \int_{\mathbb{T}} \frac{(f(x)-f(y))^2}{|x-y+2k\pi|^{1+\alpha}}dy.
\]

Let us multiply \eqref{eq1} by $u$ and define $v$ as in \eqref{eqv}. We have that
$$
\frac{1}{2}\frac{d}{dt} u^{2}+u\Lambda^\alpha u=ru^2(1-u)+\chi\left(u^2\Lambda^\beta v+\frac{1}{2}\pax(u^2)\Lambda^{\beta-1}Hv\right).
$$

Now we use identity \eqref{eqpointwise} for the dissipative term. We obtain
\begin{equation}\label{eq:ener2}
\frac{d}{dt} u^{2}+\Lambda^\alpha (u^2)+I(u)= 2ru^2(1-u)+\chi\left(2u^2\Lambda^\beta v+\pax(u^2)\Lambda^{\beta-1}Hv\right).
\end{equation}
We intend to control the term $2\chi u^2\Lambda^\beta v$ with $I (u)$. The rest on the right-hand side will turn out to be harmless. We obtain this control in a few steps.
\subsection{Auxiliary inequality for primitive function}\label{sec:auxprimi}
Let us define
\[
U (x) = \int^x_{- \pi} u (z) dz,
\]
where we suppressed the time dependence and
$$
0<s \le \frac{r}{\chi - r}.
$$ 
We have
\[
|U (x) - U (x-y) | = \left| \int^x_{x-y} u (z) dz \right| \le \| u \|_{L^{s+1}} |y|^\frac{s}{s+1}.
\]
Hence, using (3) of Lemma \ref{Lemmanorms}, we have for any $x \in \T$ and $ 0 \neq y \in \T$
\begin{equation}\label{eq:prim}
\frac{|U (x, t) - U (x-y, t) |}{ |y|^\frac{s}{2(s+1)}} \le \| u(t) \|_{L^{s+1}} |y|^\frac{s}{2(s+1)} \le e^{rt}\| u_0 \|_{L^{s+1}}   |y|^\frac{s}{2(s+1)}.
\end{equation}
Notice that the previous equation is a particular case of the splitting
$$
\frac{s}{s+1}=\left(\frac{1}{1+\delta}+\frac{\delta}{1+\delta}\right)\frac{s}{s+1}.
$$
We use this particular split $\delta=1$, because it reduces the number of parameters present in the proof and consequently makes our argument more traceable. Notice however, that the value of $\delta$ may be in fact arbitrary with no consequences to other parameters. This observation will be used in the last step  to generalize the proof.
\subsection{Towards  control of  $2\chi u^2\Lambda^\beta v$  with $I (u)$ in \eqref{eq:ener2}.}
%Recall that our aim is to control  $2\chi u^2\Lambda^\beta v$  with $I (u)$ in \eqref{eq:ener2}.
Let us compute for a $R \le 1$ 
\begin{align*}
c^{-1}_{\alpha} I (u) &\ge \;\text{P.V.} \int_{\mathbb{T}} \frac{(u(x)-u(x-y))^2}{|y|^{1+\alpha}}dy \\
&\ge  \int_{\pi >|y| > R} \frac{(u(x)-u(x-y))^2}{|y|^{1+\alpha}}dy\\
&\ge  2u^2(x) \int_{R}^\pi \frac{1}{y^{1+\alpha}}dy -2  u(x) \int_{\pi > |y| > R} \frac{u(x-y)}{|y|^{1+\alpha}}dy
\end{align*}
It reads
\begin{equation}\label{eq:I}
c^{-1}_{\alpha} I (u) \ge \frac{2}{\alpha } (R^{-\alpha} - \pi^{-\alpha} )u^2(x)  - 2  u(x) (I_1 + I_2).
\end{equation}
with 
\[
I_1 := \int_{L \ge  |y| > R} \frac{u(x-y)}{|y|^{1+\alpha}}dy, \quad I_2 := \int_{\pi > |y| \ge L} \frac{u(x-y)}{|y|^{1+\alpha}}dy,
\]
where $L \ge R$ is to be chosen later. Let us examine $I_1, I_2$. Using the identity 
$$
u(x-y) = \partial_y (U (x) -U (x-y)),
$$ 
we see that 
\begin{align*}
I_1&= (1 + \alpha)  \int_{L \ge  |y| > R} \frac{U (x, t) - U(x-y)}{|y|^{2+\alpha}} \frac{y}{|y|} dy\\
&\quad+ \frac{U (x, t) -U(x-y)}{|y|^{1+\alpha}} \bigg|^L_R + \frac{U (x, t) -U(x-y)}{|y|^{1+\alpha}} \bigg|^{-R}_{-L}.
\end{align*}
Consequently by \eqref{eq:prim} we get
\begin{align*}
I_1 &\le (1 + \alpha)   \int_{L \ge  |y| > R} \frac{  e^{rt}\| u_0 \|_{L^{s+1}}   |y|^\frac{s}{2(s+1)} }{|y|^{2+\alpha- \frac{s}{2(s+1)}}} dy\\
&\quad + 2 \frac{  e^{rt}\| u_0 \|_{L^{s+1}}   R^\frac{s}{2(s+1)}  }{R^{1+\alpha- \frac{s}{2(s+1)}}} + 2   \frac{  e^{rt}\| u_0 \|_{L^{s+1}}   L^\frac{s}{2(s+1)}  }{L^{1+\alpha- \frac{s}{2(s+1)}}}.
\end{align*}
Let us denote
$$
D :=   e^{rt}\| u_0 \|_{L^{s+1}}   L^\frac{s}{2(s+1)}.
$$
Via $D\ge e^{rt}\| u_0 \|_{L^{s+1}}   R^\frac{s}{2(s+1)}$ we arrive at
 \[
I_1  \le   (1 + \alpha) 2   \int^L_R \frac{ D}{|y|^{2+\alpha- \frac{s}{2(s+1)}}} dy + 2 \frac{  D }{R^{1+\alpha- \frac{s}{2(s+1)}}} + 2   \frac{ D}{L^{1+\alpha- \frac{s}{2(s+1)}}}
 \]
Hence we can estimate
  \[
I_1  \le  \frac{ C_0(\alpha, s) D }{R^{1+\alpha- \frac{s}{2(s+1)}}} +   \frac{C_0(\alpha, s)  D}{L^{1+\alpha- \frac{s}{2(s+1)}}},
 \]
with
 \[
C_0(\alpha, s) =  2 + \frac{2(1 + \alpha)}{2+\alpha- \frac{s}{2(s+1)}}  .
\]

For $I_2$ we use Lemma \ref{Lemmanorms} to get
\[
I_2  \le L^{-(1+\alpha)} \|u(t)\|_{L^1}\leq L^{-(1+\alpha)}\mathcal{N}.
\]
Estimates for  $I_1, I_2$ give in \eqref{eq:I}
\begin{align}\label{eq:I2}
c^{-1}_{\alpha} I (u) &\ge \frac{2}{\alpha } (R^{-\alpha} - \pi^{-\alpha} )u^2(x)  - 2  u(x) (I_1+I_2)\nonumber\\
&\ge \frac{2}{\alpha } (R^{-\alpha} - \pi^{-\alpha} )u^2(x)-2u(x)L^{-(1+\alpha)}\mathcal{N}\nonumber \\
&\quad - 2 C_0(\alpha, s) u(x) \left( \frac{  D }{R^{1+\alpha- \frac{s}{2(s+1)}}} +   \frac{ D}{L^{1+\alpha- \frac{s}{2(s+1)}}}\right).
\end{align}
Let us introduce two parameters $\gamma, m>0$ to be fixed later. Also, let us now choose $R$ and $L$ according to
\begin{equation}\label{eq:choi}
\begin{aligned}
R &:= \left(\frac{2}{\gamma\alpha} \frac{D}{u(x)+ m}\right)^{1/\alpha} \\
 D&=e^{rt}\| u_0 \|_{L^{s+1}}   L^\frac{s}{2(s+1)} := \min  \left\{e^{rt} \|u_0\|_{L^{s+1}} , c_{\alpha}\right\}.
\end{aligned}
\end{equation}
Hence
$$
R^{-\alpha}=  \frac{\gamma \alpha}{2D } (u(x)+ m)  \ge  \frac{\gamma \alpha}{2D } m
$$
and
$$
D \le   c_{\alpha}.
$$
Recall that we need $L \ge R$. The formula \eqref{eq:choi} for $D$ fixes the value for $L$, so we need to choose the appropriate value of $m$ such that $L \ge R$. 

In the case when $ \min \{e^{rt} \|u_0\|_{s+1} , c_{\alpha}\} = e^{rt} \|u_0\|_{s+1} $, $L =1$ and $R$ should verify
\begin{equation}\label{Rcondition1}
R\leq \left(\frac{2}{\gamma\alpha} \frac{D}{m}\right)^{1/\alpha}    \leq \left(\frac{2}{\gamma\alpha} \frac{  c_{\alpha}}{m}\right)^{1/\alpha}   \leq 1=L. 
\end{equation}
On the other hand, in the case of $ \min \{e^{rt} \|u_0\|_{s+1} , c_{\alpha}\} =   c_{\alpha}$, $R$ should verify
\begin{equation}\label{Rcondition2}
 L = \left( c_{\alpha} e^{-rt} \|u_0\|^{-1}_{L^{s+1}} \right)^\frac{2(s+1)}{s} \ge  \left( \frac{2}{\gamma\alpha} \frac{ c_{\alpha}}{ m} \right)^{1/\alpha}  \geq R.
\end{equation}
For equations \eqref{Rcondition1} and \eqref{Rcondition2} to hold, we take 
$$
m=  \frac{2c_\alpha}{\alpha\gamma} \max\left\{1, \left( c_{\alpha} e^{-rt} \|u_0\|^{-1}_{L^{s+1}} \right)^{-\frac{2\alpha(s+1)}{s}} \right\}. 
$$ 

The only parameter which is still free is $\gamma.$ Inserting our choice of $R$ \eqref{eq:choi} in inequality \eqref{eq:I2}, we obtain
\begin{align*}
c^{-1}_{\alpha} I (u) &\ge \frac{\gamma}{D } u^3(x)  -  \frac{2}{\alpha }\pi^{-\alpha} u^2(x)-2u(x) L^{-(1+\alpha)} \mathcal{N}   \\
&\quad- 2  u(x) C_0(\alpha, s)   \left( \frac{  D }{\left(\frac{2}{\gamma\alpha} \frac{D}{u(x)+ m}\right)^{(1+\alpha- \frac{s}{2(s+1)})/\alpha}} +   \frac{ D}{L^{1+\alpha- \frac{s}{2(s+1)}}}\right).
\end{align*}
Using $$
L^{-1}\leq \frac{e^{r\frac{2(s+1)}{s}t}}{\left( c_{\alpha} \|u_0\|^{-1}_{L^{s+1}} \right)^\frac{2(s+1)}{s}},
$$ we arrive at
\begin{align*}
 I (u) &\ge \gamma u^3(x)  -  C_1(\alpha) u^2(x)-C_2(\alpha,u_0,s)e^{(1+\alpha)r\frac{2(s+1)}{s}t} u(x)   \\
&\quad-   u(x)   \left( \frac{ 2c_\alpha C_0(\alpha, s)  D }{\left(\frac{2}{\gamma\alpha} \frac{D}{u(x)+ m}\right)^{(1+\alpha- \frac{s}{2(s+1)})/\alpha}} +    C_3(\alpha,u_0,s)e^{\left(1+\alpha- \frac{s}{2(s+1)}\right)r\frac{2(s+1)}{s}t}\right),
\end{align*}
where 
$$
C_1(\alpha)=\frac{2c_\alpha}{\alpha }\pi^{-\alpha},
$$
$$
C_2(\alpha,u_0,s)=\frac{2c_\alpha }{\left( c_{\alpha} \|u_0\|^{-1}_{L^{s+1}} \right)^{(1+\alpha)\frac{2(s+1)}{s}}} \mathcal{N} ,
$$
$$
C_3(\alpha,u_0,s)=\frac{2c_\alpha^2 C_0(\alpha,s)}{\left( c_{\alpha} \|u_0\|^{-1}_{L^{s+1}} \right)^{(1+\alpha- \frac{s}{2(s+1)})\frac{2(s+1)}{s}}}.
$$
Notice that for 
$$
1 - \alpha < \frac{s}{2(s+1)}
$$ 
the exponent ${(1+\alpha- \frac{s}{2(s+1)})/\alpha}=: \sigma <3$, so 
we arrive at
\begin{equation}\label{eq:Iq}
 I (u) \ge  \gamma u^3(x)  -  \tilde{C}(m, D, \alpha, s,  u_0) (u^{\sigma}(x) +1)-\tilde{c}(\alpha,u_0,s)e^{\bar{c}(r,s,\alpha)t}u(x)
\end{equation}
with 
$$
\sigma<3.
$$

\subsection{Control of $2\chi u^2\Lambda^\beta v$  with $I (u)$ in \eqref{eq:ener2}.}
Let us use \eqref{eq:Iq} in \eqref{eq:ener2}. Hence, we obtain for $\rho = u^2$
\begin{align}\label{eq:ener2qb}
\pat \rho+\Lambda^\alpha \rho+(\gamma +2r)  \rho^\frac{3}{2}&\leq 2r\rho+\chi\left(2\rho \Lambda^\beta v+\pax \rho \Lambda^{\beta-1}Hv\right)\nonumber\\
&\quad+\tilde{C}(m, D, \alpha, s,  u_0) (\rho^{\sigma/2}(x) +1)\nonumber\\
&\quad+\tilde{c}(\alpha,u_0,s)e^{\bar{c}(r,s,\alpha)t}\sqrt{\rho}.
\end{align}
Since Lemma \ref{lemmaboundsv} (part 1) and Definition \eqref{eqvB2} together with the positivity of $u$, we have
\begin{equation}\label{eq:pos}
\Lambda^\beta v\leq u,
\end{equation}
for both operators $B(u)$ considered.

As a consequence, equation \eqref{eq:ener2qb} can be estimated as
\begin{align}\label{eq:ener2q}
\pat \rho+\Lambda^\alpha \rho+(\gamma +2r - 2\chi)  \rho^\frac{3}{2}&\leq 2r\rho+\chi\pax \rho \Lambda^{\beta-1}Hv\nonumber\\
&\quad+\tilde{C}(m, D, \alpha, s,  u_0) (\rho^{\sigma/2}(x) +1)\nonumber\\
&\quad+\tilde{c}(\alpha,u_0,s)e^{\bar{c}(r,s,\alpha)t}\sqrt{\rho}.
\end{align}
We are free to choose the value of $\gamma$. Let us take $\gamma=2\chi$. Now, defining
$$
\bar{\rho} (t) = \max_{x \in \T}  \rho (x, t)=\rho(x_t,t)
$$
we have
\begin{align*}
\frac{d}{dt} \bar{\rho} (t)+\Lambda^\alpha \rho(x_t)+ 2r \bar{\rho} ^\frac{3}{2}&\leq 2r\bar{\rho}+\tilde{C}(m, D, \alpha, s,  u_0) (\bar{\rho}^{\sigma/2}(x) +1)\\
&\quad+\tilde{c}(\alpha,u_0,s)e^{\bar{c}(r,s,\alpha)t}\sqrt{\bar{\rho}}.
\end{align*}
Thanks to $r>0$ and $\sigma<3$, we get
$$
\tilde{C}(m, D, \alpha, s,  u_0) (\bar{\rho}^{\sigma/2}(x) +1)-2r\bar{\rho}^{3/2}\leq \bar{C}(m, D, \alpha, s,  u_0).
$$
Therefore
\[
\frac{d}{dt} \bar{\rho}+\Lambda^\alpha \rho(x_t)+ \leq 2r\bar{\rho}+\bar{C}(m, D, \alpha, s,  u_0) 
+\tilde{c}(\alpha,u_0,s)e^{\bar{c}(r,s,\alpha)t}\sqrt{\bar{\rho}},
\]
which implies that
$$
\sup_{0\leq t\leq T}\|u(t)\|_{L^\infty}\leq C(r,s,\alpha,u_0,T,m,D)<\infty.
$$
Notice that $T$ was fixed at the beginning of the proof but it can be taken arbitrarily large.
\subsection{Recovering the range of admissible $\alpha$'s}
Observe that the required assumptions were 
\begin{itemize}
\item $
s \le \frac{r}{\chi - r}
$
(for the $L^{1+s}$ estimate in Lemma \ref{Lemmanorms}) and 
\item
$
1 - \alpha < \frac{s}{2(s+1)}
$
(for  \eqref{eq:Iq}).
\end{itemize}
In particular, we can take $s=r/(\chi-r)$ and we get the global existence of solution in the range
\begin{equation}\label{range}
\alpha\in \left(1- \frac{1}{2} \frac{r}{\chi},2\right).
\end{equation}

\subsection{Generalizing the range of admissible $\alpha$'s}
Observe that the expression 
$$
\frac{1}{2}\frac{r}{\chi}
$$
in \eqref{range} originated in Section \ref{sec:auxprimi}. In particular, we can change \eqref{eq:prim} as follows
\begin{equation}\label{eq:prim2}
\frac{|U (x, t) - U (x-y, t) |}{ |y|^{\frac{1}{1+\delta}\frac{s}{(s+1)}}} \le \| u(t) \|_{L^{s+1}} |y|^{\delta/(1+\delta)} \le e^{rt}\| u_0 \|_{L^{s+1}}   |y|^{\delta/(1+\delta)}
\end{equation}
for an arbitrary $\delta>0$. By doing so, we arrive at the range
\begin{equation}\label{rangefull}
\alpha\in \left(1- \frac{1}{1+\delta}\frac{r}{\chi},2\right).
\end{equation}
Since we can take $0<\delta\ll1$, we recover the range as in the assumption. \qed
\section{Proof of Proposition \ref{th:largetime}}\label{sec:7}
Define
$$
\bar{u}(t)=\max_{x}u(x,t)=u(x_t,t).
$$
Then we have
$$
\frac{d}{dt}\bar{u}+\Lambda^\alpha u(x_t)=\chi\bar{u}\Lambda^\beta v(x_t)+r\bar{u}(1-\bar{u}).
$$
Since \eqref{eq:pos} holds for $v$ defined in \eqref{eqvB} and \eqref{eqvB2}, we obtain
$$
\frac{d}{dt}\bar{u}+\Lambda^\alpha u(x_t)\leq \chi\bar{u}^2+r\bar{u}(1-\bar{u}).
$$
Let
$$
r_0=\frac{2\mathcal{N}}{u (x_t)}.
$$
We have two possibilities, namely $r_0\geq\pi$ or $r_0\leq \pi$. In the former case  we have the bound
$$
\frac{2\mathcal{N}}{\pi}\geq \|u (t)\|_{L^\infty}.
$$
On the other hand, if $r_0 \leq \pi$, we can apply Lemma \ref{lemaaux3}  with $p=1, \gamma_1 = \mathcal{N}$  and get
$$
\Lambda^\alpha u(x_t)\geq \frac{c_\alpha}{2^{\alpha}}\frac{\bar{u}^{1+\alpha }}{\mathcal{N}^{\alpha }}.
$$
Thus for $\frac{2\mathcal{N}}{\pi} \leq \|u (t)\|_{L^\infty}$ holds
$$
\frac{d}{dt}\bar{u}\leq \chi\bar{u}^2+r\bar{u}(1-\bar{u})-\frac{c_\alpha}{2^{\alpha}}\frac{\bar{u}^{1+\alpha }}{\mathcal{N}^{\alpha }}.
$$
This is a differential inequality of type
\begin{equation}\label{odiP}
\frac{d}{dt} X (t) \leq A X(t) + B X^2 (t) - C X^{1 + \alpha} (t)
\end{equation}
Let us focus on the case $\alpha>1$. We use a blowup argument to obtain the global bound. More precisely, let us denote by $s_0$ the lower bound of values $s$ for which $A s + B s^2 - C s^{1 + \alpha} < -1$. 
Consider the case when $X(0) <  \max(s_0, \frac{2\mathcal{N}}{\pi} )$. Assuming that there exists first time $t_0>0$ such that  $X(0) = \max(s_0, \frac{2\mathcal{N}}{\pi} )$, we obtain from \eqref{odiP} that $\frac{d}{dt} X (t) \leq -1$, which contradicts the fact that $t_0$ is the first time of equality. As a consequence, $X(t) < \max(s_0, \frac{2\mathcal{N}}{\pi} )$ for all times. The remaining case $X(0) \ge \max(s_0, \frac{2\mathcal{N}}{\pi} )$ means that as long as  $X(t) \ge \max(s_0, \frac{2\mathcal{N}}{\pi} )$, we can use \eqref{odiP} that gives exponential damping  $\frac{d}{dt} X (t) \leq -1$. Therefore at certain $t_1< \infty$ we have $X(t_1) <  \max(s_0, \frac{2\mathcal{N}}{\pi} )$. Now we repeat the argument from the previous case.

Thus, we obtain the global bound
$$
\sup_{0\leq t<\infty}\|u(t)\|_{L^\infty}\leq C(\alpha,r,\chi,  \| u_0 \|_{L^\infty},\| u_0 \|_{L^1}).
$$
In the 'critical' case $\alpha=1$, we have for $\frac{2\mathcal{N}}{\pi} \leq \|u (t)\|_{L^\infty}$
$$
\frac{d}{dt}\bar{u}\leq \left(\chi-r-\frac{1}{2\pi \max(\| u_0\|_{L^1}, 2 \pi )}\right)\bar{u}^2+r\bar{u}.
$$
Under the assumption $\chi-r-\frac{1}{2\pi \max(\| u_0\|_{L^1}, 2 \pi)} <0$, the blowup argument implies again the thesis.
%$$
%\sup_{0\leq t\leq T}\|u(t)\|_{L^\infty}\leq \max\{\|u_0\|_{L^\infty},1\}.
%$$
\qed

\section*{Acknowledgments} JB is partially supported by the National Science Centre (NCN) grant no. 2011/01/N/ST1/05411. RGB is partially supported by the Department of Mathematics at University of California, Davis.

\appendix
\section{Pointwise estimates for the fractional laplacian}
\begin{lem}\label{lemaaux3}
Let $h\in C^2(\TT)$ be a positive function and write $h(x^*)=\max_x h(x)=\|h\|_{L^\infty}$.
Then, if $h$ verifies the bounds
$$
\|h\|_{L^\infty}/2\geq\langle h \rangle,\,\|h\|_{L^p(\TT)}\leq \gamma_p,\text{ with any }\gamma_p\leq\frac{\sqrt[p]{\pi}\|h\|_{L^\infty}}{2},
$$
then
$$
\Lambda^\alpha h(x^*)\geq \frac{\Gamma(1+\alpha)\cos((1-\alpha)\pi/2)}{\pi}\frac{1}{2^{p\alpha}}\frac{\|h\|_{L^\infty}^{1+\alpha p}}{\gamma_p^{\alpha p}}.
$$
\end{lem}
\begin{proof} Let $\rho >0$ be a constant to be fixed later. We define 
$$
\mathcal{U}_1=\{\eta\in B(0,\rho) \; | \; h(x^*)-h(x^*-\eta)>h(x^*)/2 \},
$$
and $\mathcal{U}_2=B(0,\rho)-\mathcal{U}_1$. Notice that if the function is sharp enough, \emph{i.e.} if $h(x^*)/2 > \min_{x \in B(0,\rho)}  h(x) $, then $\mathcal{U}_1\neq\emptyset$. Starting from our assumption, we have
$$
\gamma^p_p\geq\|h\|_{L^p}^p=\int_\TT|h(x^*-\eta)|^p d\eta\geq \int_{\mathcal{U}_2}|h(x^*-\eta)|^pd\eta\geq\frac{ |h(x^*)|^p}{2^p}|\mathcal{U}_2|,
$$
so, via positivity
\begin{equation}\label{eqappaux}
-\left(\frac{2\gamma_p}{h(x^*)}\right)^p = -\left(\frac{2\gamma_p}{|h(x^*)|}\right)^p\leq -|\mathcal{U}_2|.
\end{equation}
Recalling \eqref{Lkernelalpha} together with 
$$
c(\alpha)=\frac{\Gamma(1+\alpha)\cos((1-\alpha)\pi/2)}{\pi},
$$
we have
\begin{eqnarray*}
\Lambda^\alpha h(x^*)&=&c(\alpha)\sum_{k}\text{P.V.}\int_\TT\frac{h(x^*)-h(x^*-\eta)}{|\eta+k2\pi|^{1+\alpha}}d\eta\\
&\geq& c(\alpha)\text{P.V.}\int_{\mathcal{U}_1}\frac{h(x^*)-h(x^*-\eta)}{|\eta|^{1+\alpha}}d\eta\\
&\geq& c(\alpha)\frac{h(x^*)}{2 \rho^{1+\alpha}}|\mathcal{U}_1|\\
&\geq& c(\alpha)\frac{h(x^*)}{2 \rho^{1+\alpha}}\left(2 \rho-|\mathcal{U}_2|\right)\\
&\geq& c(\alpha)\frac{h(x^*)}{2 \rho^{1+\alpha}}\left(2 \rho -\left(\frac{2\gamma_p}{h(x^*)}\right)^p\right).
\end{eqnarray*}
Let us fix now
$$
\rho=\left(\frac{2\gamma_p}{h(x^*)}\right)^p,
$$
thus
$$
\Lambda^\alpha h(x^*)\geq c(\alpha)\frac{h(x^*)\left(\frac{2\gamma_p}{h(x^*)}\right)^p}{\left(\frac{2\gamma_p}{h(x^*)}\right)^{p+p\alpha}}=\frac{c(\alpha)}{2^{p\alpha}}\frac{h(x^*)^{1+\alpha p}}{\gamma_p^{\alpha p}}.
$$
Finally notice that due to the boundedness of the domain we have to impose the restriction
$$
r=\left(\frac{2\gamma_p}{h(x^*)}\right)^p\leq \pi \text{ i.e. } \gamma_p\leq\frac{\sqrt[p]{\pi}\|h\|_{L^\infty}}{2}.
$$
\end{proof}

\section{Sobolev-type inequalities for the fractional laplacian}
We have the following Sobolev-type bounds for the entropy
\begin{lem}\label{lemaentropy}
Let $0\leq u\in L^p(\TT^d)$ be a given function and $0<\alpha<2$, $0<\delta<\alpha/2$ two fixed constants. Then
\begin{itemize}
\item for $p=1$
$$
\|u\|_{\dot{W}^{\alpha/2-\delta,1}}^2\leq C(\alpha,d,\delta)\|u\|_{L^1}\int_{\TT^d}\Lambda^\alpha u(x)\log(u(x))dx,
$$
\item for $p=\infty$
$$
\|u\|_{\dot{H}^{\alpha/2}}^2\leq C(\alpha,d)\|u\|_{L^\infty}\int_{\TT^d}\Lambda^\alpha u(x)\log(u(x))dx.
$$
\end{itemize}
\end{lem}
For the proof of the case $p=1$, the interested reader can see \cite{BGL} by the authors \& Luli. The case $p=\infty$ was proved by Bae \& Granero-Belinch\'on in \cite{bae2015global}.

In the next lemma we obtain fractional Sobolev bounds in terms of $\int u^s\Lambda^\alpha$:
\begin{lem}\label{lemaentropy2}
Let $0\leq u\in L^{1+s}(\TT)$, $s\le 1$, be a given function and $0<\alpha<2$, $0<\delta<\alpha/(2+2s)$ two fixed constants. Then, 
$$
0 \le \int_{\TT}\Lambda^\alpha u(x) u^s(x)dx.
$$
Moreover,
$$
\|u\|_{\dot{W}^{\alpha/(2+2s)-\delta,1+s}}^{2+2s}\leq C(\alpha,s,\delta)\|u\|_{L^{1+s}}^{1+s}\int_{\TT}\Lambda^\alpha u(x) u^s(x)dx.
$$
\end{lem}
\begin{proof}
Let us define
$$
I_1=\int_\T u^s(x)  \Lambda^\alpha u(x) dx.
$$
Using \eqref{Lkernelalpha} and changing variables, we compute
\begin{align*}
I_1&=c_\alpha\int_\T\sum_{k\in \ZZ, k\neq 0}\int_{\TT} u^s(x) \frac{u(x)-u(\eta)}{|x-\eta+2k\pi|^{1+\alpha}} d\eta dx\\
&\quad +c_\alpha\int_\T\text{P.V.}\int_{\TT} u^s(x)\frac{u(x)-u(\eta)}{|x-\eta|^{1+\alpha}} d\eta dx\\
&=c_\alpha\int_\T\sum_{k\in \ZZ, k\neq 0}\int_{\TT} u^s(\eta) \frac{u(\eta)-u(x)}{|\eta-x+2k\pi|^{1+\alpha}} d\eta dx\\
&\quad +c_\alpha\int_\T\text{P.V.}\int_{\TT} u^s(\eta)\frac{u(\eta)-u(x)}{|x-\eta|^{1+\alpha}}d\eta dx.
\end{align*}
Hence
\begin{align*}
I_1&=\frac{c_\alpha}{2}\int_\T\sum_{k\in \ZZ, k\neq 0}\int_{\TT} (u^s(x)-u^s(\eta)) \frac{u(x)-u(\eta)}{|x-\eta+2k\pi|^{1+\alpha}} d\eta dx\\
&\quad +\frac{c_\alpha}{2}\int_\T\text{P.V.}\int_{\TT} (u^s(x)-u^s(\eta))\frac{u(x)-u(\eta)}{|x-\eta|^{1+\alpha}} d\eta dx\\
&\geq0.
\end{align*}
In particular, using $u \ge 0$
%\begin{equation}
\begin{align}\label{L7I1}
I_1&\geq\frac{c_\alpha}{2}\int_\T\int_{\TT} (u^s(x)-u^s(\eta))\frac{u(x)-u(\eta)}{|x-\eta|^{1+\alpha}} d\eta dx\nonumber\\
&=  \frac{c_\alpha}{2}\int_\T\int_{\TT}\int_0^1 \frac{d}{d \lambda}  \left( (\lambda u(x)+(1-\lambda)u(\eta))^s \right)     \frac{u(x)-u(\eta)}{|x-\eta|^{1+\alpha}} d\eta dx \nonumber \\
& =   \frac{c_\alpha}{2}\int_\T\int_{\TT}\int_0^{1} \frac{s}{(\lambda u(x)+(1-\lambda)u(\eta))^{1-s}}\frac{(u(x)-u(\eta))^2}{|x-\eta|^{1+\alpha}} d \lambda d\eta dx.
\end{align}
%\end{equation}
Let us define $\beta$ and $I$
$$
\beta=\frac{\alpha}{2+2s}-\delta,
$$
$$
I=\|u\|_{\dot{W}^{\beta,1+s}}^{1+s}=\int_\TT\int_\TT\frac{|u(x)-u(\eta)|^{1+s}}{|x-\eta|^{1+(1+s)\beta}}dxd\eta.
$$
Then we compute
\begin{align*}
I&=\int_\TT\int_\TT\int^1_0\frac{|u(x)-u(\eta)|^{1+s}}{|x-\eta|^{1+(1+s)\beta}}d\lambda dxd\eta\\
&=\int_\TT\int_\TT\int^1_0\frac{|u(x)-u(\eta)|}{|x-\eta|^{0.5+(1+s)\beta-\alpha/2}}\frac{|u(x)-u(\eta)|^s}{|x-\eta|^{0.5+\alpha/2}} \frac{|\lambda u(x)+(1-\lambda)u(\eta)|^{(1-s)/2}}{|\lambda u(x)+(1-\lambda)u(\eta)|^{(1-s)/2}} d\lambda dxd\eta\\
&=\int_\TT\int_\TT\int^1_0 F(x,\eta,\lambda)G(x,\eta,\lambda)d\lambda dx d\eta,
\end{align*}
where
$$
F=\frac{|u(x)-u(\eta)|}{|x-\eta|^{0.5+\alpha/2}}\frac{1}{|\lambda u(x)+(1-\lambda)u(\eta)|^{(1-s)/2}},
$$
$$
G=\frac{|u(x)-u(\eta)|^s}{|x-\eta|^{0.5+(1+s)\beta-\alpha/2}}|\lambda u(x)+(1-\lambda)u(\eta)|^{(1-s)/2}.
$$
Consequently 
\begin{equation}\label{L7I}
I\leq \|F\|_{L^2(\TT\times\TT\times[0,1])}\|G\|_{L^2(\TT\times\TT\times[0,1])}.
\end{equation}
We have via \eqref{L7I1}
\[
\|F\|_{L^2(\TT\times\TT\times[0,1])}^2=\int_\TT\int_\TT \int_0^1\frac{(u(x)-u(\eta))^2}{|x-\eta|^{1+\alpha}}\frac{d\lambda dxd\eta}{(\lambda u(x)+(1-\lambda)u(\eta))^{1-s}} \leq \frac{2}{c_\alpha s} I_1.
\]
Since
\begin{align*}
G^2&=\frac{|u(x)-u(\eta)|^{2s}}{|x-\eta|^{1+2(1+s)\beta-\alpha}}|\lambda u(x)+(1-\lambda)u(\eta)|^{1-s}\\
&\leq C(s)\frac{|u(x)|^{1+s}+|u(\eta)|^{1+s}}{|x-\eta|^{1+2(1+s)\beta-\alpha}} =  C(s)\frac{|u(x)|^{1+s}+|u(\eta)|^{1+s}}{|x-\eta|^{1- 2(1+s) \delta}},
\end{align*}
thus
$$
\|G\|_{L^2(\TT\times\TT\times[0,1])}^2\leq C(s,\delta)\|u\|_{L^{1+s}}^{1+s}.
$$
Estimates for $F$ and $G$ in \eqref{L7I} give thesis.
\end{proof}


\begin{thebibliography}{00}

\bibitem{AGM}
Y.~Ascasibar, R.~Granero-Belinch\'on, and J.~M. Moreno.
\newblock An approximate treatment of gravitational collapse.
\newblock {\em Physica D: Nonlinear Phenomena}, 262:71 -- 82, 2013.

\bibitem{Atk}
R.~Atkinson, C.~Rhodes, D.~Macdonald, and R.~Anderson.
\newblock Scale-free dynamics in the movement patterns of jackals.
\newblock {\em Oikos}, 98(1):134--140, 2002.

\bibitem{bae2015global}
H.~Bae and R.~Granero-Belinch{\'o}n.
\newblock Global existence for some transport equations with nonlocal velocity.
\newblock {\em Advances in Mathematics}, 269:197--219, 2015.

\bibitem{Bart03}
F.~Bartumeus, F.~Peters, S.~Pueyo, C.~Marras{\'e}, and J.~Catalan.
\newblock Helical l{\'e}vy walks: adjusting searching statistics to resource
  availability in microzooplankton.
\newblock {\em Proceedings of the National Academy of Sciences},
  100(22):12771--12775, 2003.

\bibitem{Bi1}
P.~Biler.
\newblock {G}rowth and accretion of mass in an astrophysical model.
\newblock {\em Appl. Math.(Warsaw)}, 23(2):179--189, 1995.

\bibitem{BilKar10}
P.~Biler, G. Karch. Blowup of solutions to generalized Keller-Segel model. {\em J. Evol. Equ.}, 10, no. 2, 247 -- 262, 2010

\bibitem{biler2010nonlinear}
P.~Biler, G.~Karch, and R.~Monneau.
\newblock Nonlinear diffusion of dislocation density and self-similar
  solutions.
\newblock {\em Communications in Mathematical Physics}, 294(1):145--168, 2010.

\bibitem{BilerWu}
P.~Biler and G.~Wu.
\newblock Two-dimensional chemotaxis models with fractional diffusion.
\newblock {\em Math. Methods Appl. Sci.}, 32(1):112--126, 2009.

\bibitem{blanchet2011parabolic}
A.~Blanchet.
\newblock {O}n the parabolic-elliptic {P}atlak-{K}eller-{S}egel system in
  dimension 2 and higher.
\newblock {\em {S}{\'e}minaire {L}aurent {S}chwartz - {E}{D}{P} et
  applications}, (8), 2011.

\bibitem{bournaveas2010one}
N.~Bournaveas and V.~Calvez.
\newblock The one-dimensional {K}eller-{S}egel model with fractional diffusion
  of cells.
\newblock {\em {N}onlinearity}, 23(4):923, 2010.

\bibitem{BG}
J.~Burczak and R.~Granero-Belinch\'on.
\newblock Boundedness of large-time solutions to a chemotaxis model with
  nonlocal and semilinear flux.
\newblock {\em To appear in {T}opological {M}ethods in {N}onlinear {A}nalysis.
  Arxiv Preprint arXiv:1409.8102 [math.AP]}.

\bibitem{BG3}
J.~Burczak and R.~Granero-Belinch\'on.
\newblock Critical {K}eller-{S}egel meets {B}urgers on $\mathbb{S}^1$.
\newblock {\em Submitted. Arxiv Preprint arXiv:1504.00955 [math.AP]}.

\bibitem{BG2}
J.~Burczak and R.~Granero-Belinch\'on.
\newblock On a generalized, doubly parabolic {K}eller-{S}egel system in one
  spatial dimension.
\newblock {\em Submitted. Arxiv Preprint arXiv:1407.2793 [math.AP]}.

\bibitem{BGL}
J.~Burczak,  R.~Granero-Belinch\'on and Garving K. Luli.
\newblock On the generalized Buckley-Leverett equation.
\newblock {\em Preprint}.


\bibitem{Carrillo}
J.~A. Carrillo, L.~C.~F. Ferreira, and J.~C. Precioso.
\newblock A mass-transportation approach to a one dimensional fluid mechanics
  model with nonlocal velocity.
\newblock {\em Adv. Math.}, 231(1):306--327, 2012.

\bibitem{CC}
A.~Castro and D.~C{\'o}rdoba.
\newblock {G}lobal existence, singularities and ill-posedness for a nonlocal
  flux.
\newblock {\em Advances in Mathematics}, 219(6):1916--1936, 2008.

\bibitem{CC2}
A.~Castro and D.~C{\'o}rdoba.
\newblock {S}elf-similar solutions for a transport equation with non-local
  flux.
\newblock {\em Chinese Annals of Mathematics, Series B}, 30(5):505--512, 2009.

\bibitem{ccfgl}
A.~Castro, D.~Cordoba, C.~Fefferman, F.~Gancedo, and M.~Lopez-Fernandez.
\newblock {R}ayleigh-{T}aylor breakdown for the {M}uskat problem with
  applications to water waves.
\newblock {\em Annals of Math}, 175:909--948, 2012.

\bibitem{CCCF}
D.~Chae, A.~C{\'o}rdoba, D.~C{\'o}rdoba, and M.~A. Fontelos.
\newblock {F}inite time singularities in a {1D} model of the quasi-geostrophic
  equation.
\newblock {\em Advances in Mathematics}, 194(1):203--223, 2005.

\bibitem{Cole}
B.~J. Cole.
\newblock Fractal time in animal behaviour: the movement activity of
  drosophila.
\newblock {\em Animal Behaviour}, 50(5):1317--1324, 1995.

\bibitem{ccgs-10}
P.~Constantin, D.~Cordoba, F.~Gancedo, and R.~Strain.
\newblock On the global existence for the {M}uskat problem.
\newblock {\em Journal of the European Mathematical Society}, 15:201--227,
  2013.

\bibitem{constantin2013long}
P.~Constantin, A.~Tarfulea, and V.~Vicol.
\newblock Long time dynamics of forced critical sqg.
\newblock {\em Communications in Mathematical Physics}, pages 1--49, 2013.

\bibitem{constantin2012nonlinear}
P.~Constantin and V.~Vicol.
\newblock Nonlinear maximum principles for dissipative linear nonlocal
  operators and applications.
\newblock {\em Geometric And Functional Analysis}, 22(5):1289--1321, 2012.

\bibitem{cordoba2003pointwise}
A.~C{\'o}rdoba and D.~C{\'o}rdoba.
\newblock {A} pointwise estimate for fractionary derivatives with applications
  to partial differential equations.
\newblock {\em Proceedings of the National Academy of Sciences}, 100(26):15316,
  2003.

\bibitem{cor2}
A.~C{\'o}rdoba and D.~C{\'o}rdoba.
\newblock A maximum principle applied to quasi-geostrophic equations.
\newblock {\em Communications in Mathematical Physics}, 249(3):511--528, 2004.

\bibitem{gazolaz2005formation}
A.~C{\'o}rdoba, D.~C{\'o}rdoba, and M.~A. Fontelos.
\newblock Formation of singularities for a transport equation with nonlocal
  velocity.
\newblock {\em Annals of mathematics}, 162(3):1375--1387, 2005.

\bibitem{c-g09}
D.~C{\'o}rdoba and F.~Gancedo.
\newblock {A} maximum principle for the {M}uskat problem for fluids with
  different densities.
\newblock {\em Communications in Mathematical Physics}, 286(2):681--696, 2009.

\bibitem{c-g-o08}
D.~C{\'o}rdoba, F.~Gancedo, and R.~Orive.
\newblock {A} note on interface dynamics for convection in porous media.
\newblock {\em Physica D: Nonlinear Phenomena}, 237(10-12):1488--1497, 2008.

\bibitem{CGO}
D.~C\'ordoba, R.~Granero-Belinch\'on, and R.~Orive.
\newblock {O}n the confined {M}uskat problem: differences with the deep water
  regime.
\newblock {\em Communications in Mathematical Sciences}, 12(3):423--455, 2014.

\bibitem{DKSV}
M.~Dabkowski, A.~Kiselev, L.~Silvestre, and V.~Vicol.
\newblock Global well-posedness of slightly supercritical active scalar
  equations.
\newblock {\em Analysis and PDE}, 7(1):43--72, 2014.

\bibitem{dislocation1}
J.~Deslippe, R.~Tedstrom, M.~Daw, D.~Chrzan, T.~Neeraj, and M.~Mills.
\newblock Dynamics scaling in a simple one-dimensional model of dislocation
  activity.
\newblock {\em Philosophical Magazine}, 84:2445--2454, 2004.

\bibitem{HDong}
H.~Dong.
\newblock Well-posedness for a transport equation with nonlocal velocity.
\newblock {\em Journal of Functional Analysis}, 255(11):3070--3097, 2008.

\bibitem{dong2009finite}
H.~Dong, D.~Du, and D.~Li.
\newblock Finite time singularities and global well-posedness for fractal
  burgers equations.
\newblock {\em Indiana University mathematics journal}, 58(2):807--821, 2009.

\bibitem{dong2014one}
H.~Dong and D.~Li.
\newblock On a one-dimensional $\alpha$-patch model with nonlocal drift and
  fractional dissipation.
\newblock {\em Transactions of the American Mathematical Society},
  366(4):2041--2061, 2014.

\bibitem{escudero2006fractional}
C.~Escudero.
\newblock The fractional {K}eller-{S}egel model.
\newblock {\em Nonlinearity}, 19(12):2909, 2006.

\bibitem{G}
R.~Granero-Belinch{\'o}n.
\newblock Global existence for the confined muskat problem.
\newblock {\em SIAM Journal on Mathematical Analysis}, 46(2):1651--1680, 2014.

\bibitem{GH}
R.~Granero-Belinch\'on and J.~Hunter.
\newblock On a nonlocal analog of the kuramoto-sivashinsky equation.
\newblock {\em Nonlinearity}, 28(4):1103--1133, 2015.

\bibitem{GNO}
R.~Granero-Belinch\'on, G.~Navarro, and A.~Ortega.
\newblock On the effect of boundaries in two-phase porous flow.
\newblock {\em Nonlinearity}, 28(2):435--461, 2015.

\bibitem{GO}
R.~Granero-Belinch{\'o}n and R.~Orive-Illera.
\newblock An aggregation equation with a nonlocal flux.
\newblock {\em Nonlinear Analysis: Theory, Methods \& Applications}, 108(0):260
  -- 274, 2014.

\bibitem{Hillen3}
T.~Hillen and K.~J. Painter.
\newblock A user's guide to {PDE} models for chemotaxis.
\newblock {\em J. Math. Biol.}, 58(1-2):183--217, 2009.

\bibitem{keller1970initiation}
E.~Keller and L.~Segel.
\newblock {I}nitiation of slime mold aggregation viewed as an instability.
\newblock {\em Journal of Theoretical Biology}, 26(3):399--415, 1970.

\bibitem{Kiselev}
A.~Kiselev.
\newblock Regularity and blow up for active scalars.
\newblock 5:225--255, 2010.

\bibitem{kiselevburgers}
A.~Kiselev, F.~Nazarov, and R.~Shterenberg.
\newblock Blow up and regularity for fractal {B}urgers equation.
\newblock {\em Dyn. Partial Differ. Equ.}, 5(3):211--240, 2008.

\bibitem{Klaf90}
J.~Klafter, B.~White, and M.~Levandowsky.
\newblock Microzooplankton feeding behavior and the levy walk.
\newblock In {\em Biological motion}, pages 281--296. Springer, 1990.

\bibitem{Lew_nencki}
M.~Levandowsky, B.~White, and F.~Schuster.
\newblock Random movements of soil amebas.
\newblock {\em Acta Protozoologica}, 36:237--248, 1997.

\bibitem{LiRodrigo2}
D.~Li and J.~Rodrigo.
\newblock Blow-up of solutions for a 1{D} transport equation with nonlocal
  velocity and supercritical dissipation.
\newblock {\em Adv. Math.}, 217(6):2563--2568, 2008.

\bibitem{LiRodrigo}
D.~Li and J.~L. Rodrigo.
\newblock On a one-dimensional nonlocal flux with fractional dissipation.
\newblock {\em SIAM J. Math. Anal.}, 43(1):507--526, 2011.

\bibitem{bertozzi-Majda}
A.~Majda and A.~Bertozzi.
\newblock {\em {V}orticity and incompressible flow}.
\newblock Cambridge Univ Pr, 2002.

\bibitem{Morlet}
A.~Morlet.
\newblock Further properties of a continuum of model equations with globally
  defined flux.
\newblock {\em Journal of Mathematical Analysis and Applications},
  221:132--160, 1998.

\bibitem{hadza}
D.~A. Raichlen, B.~M. Wood, A.~D. Gordon, A.~Z. Mabulla, F.~W. Marlowe, and
  H.~Pontzer.
\newblock Evidence of l{\'e}vy walk foraging patterns in human
  hunter--gatherers.
\newblock {\em Proceedings of the National Academy of Sciences},
  111(2):728--733, 2014.

\bibitem{Shl86}
M.~F. Shlesinger and J.~Klafter.
\newblock L{\'e}vy walks versus l{\'e}vy flights.
\newblock In {\em On growth and form}, pages 279--283. Springer, 1986.

\bibitem{silvestre2014transport}
L.~Silvestre and V.~Vicol.
\newblock On a transport equation with nonlocal drift.
\newblock {\em To appear in Transaction of the American Mathematical Society,
  arXiv preprint arXiv:1408.1056}, 2014.

\bibitem{WuZheng}
G.~Wu and X.~Zheng.
\newblock On the well-posedness for {K}eller-{S}egel system with fractional
  diffusion.
\newblock {\em Math. Methods Appl. Sci.}, 34(14):1739--1750, 2011.
\end{thebibliography}
\end{document}